\documentclass[12pt]{article}

\usepackage{amsfonts}
\usepackage{amssymb}
\usepackage{amsthm}
\usepackage{amsmath}
\usepackage{mathrsfs}
\usepackage{xcolor}
\usepackage[all]{xy}
\usepackage[utf8]{inputenc}

\def\uno{\mathbf{1}}
\def\mpn{\medskip\par\noindent}

\usepackage[left=1.20in,right=1.20in]{geometry}
\definecolor{ao(english)}{rgb}{0.0, 0.5, 0.0}
\definecolor{brickred}{rgb}{0.8, 0.25, 0.33}
\definecolor{burntorange}{rgb}{0.8, 0.33, 0.0}
\definecolor{beaver}{rgb}{0.62, 0.51, 0.44}
\definecolor{brown(traditional)}{rgb}{0.59, 0.29, 0.0}
\definecolor{ao(english)}{rgb}{0.0, 0.5, 0.0}
\definecolor{verde}{rgb}{0.12, 0.8, 0.17}
\def\rouge{\color{red}}

\newcommand{\rac}{\mathbb Q}

\newcommand{\lae}{\varepsilon}

\newcommand{\Ker}{\mathrm{Ker}}

\newcommand{\mult}{\times\,\_\,}

\theoremstyle{plain}
\newtheorem{teo}{Theorem}[section]
\newtheorem*{teo-non}{Theorem}
\newtheorem{prop}[teo]{Proposition}

\newtheorem{lema}[teo]{Lemma}

\theoremstyle{definition}
\newtheorem{defi}[teo]{Definition}
\newtheorem{nota}[teo]{Notation}

\theoremstyle{remark}
\newtheorem{ejem}[teo]{Example}

\newtheorem{rem}[teo]{Remark}
\def\CC{\mathcal{C}}
\def\CD{\mathcal{D}}
\def\CF{\mathcal{F}}
\def\CL{\mathcal{L}}

\def\CR{\mathcal{R}}

\def\mpn{\medskip\par\noindent}

\def\N{\mathbb{N}}

\def\endpf{~\leaders\hbox to 1em{\hss\  \hss}\hfill~\raisebox{.5ex}{\framebox[1ex]{}}\smallskip\par}

\newcommand\Rres[2]{\mathcal{R}^{#1}_{#2}}
\reversemarginpar
\author{Nadia Romero\footnote{\texttt{nadia.romero@ugto.mx}}\\ 
\begin{small}
Departamento de Matem\'aticas,
\end{small}\\
\begin{small}
Universidad de Guanajuato.
\end{small}
}

\author{Serge Bouc and Nadia Romero}

\title{Koszul resolution for linear monoidal functors}

\date{}

\begin{document}
\maketitle
\begin{abstract}
We introduce regular sequences and associated Koszul resolutions for monoids in the category of functors over  an essentially small linear symmetric monoidal category. Next we define polynomials over such monoids. We compute the Hochschild cohomology functors and prove a relative analogue of Hilbert's syzygy theorem for polynomials over {\em tensor idempotent commutative monoids}.\smallskip
\end{abstract}

\section{Introduction}

Hochschild cohomology is a classical tool to study rings and bimodules over them. It can be generalized to other contexts in various ways: e.g. in~\cite{primero}, the second author introduces Hochschild cohomology for functors over linear symmetric monoidal categories. It is then a natural question to try to generalize to this framework some other notions. In the classical framework of rings, the construction of {\em Koszul resolutions} plays a fundamental r\^ole. Such resolutions can be used in particular to give a proof of the Hilbert syzygy theorem for polynomial rings over a field (see Application 4.5.6 in~\cite{weibel}).\par
In this work, we introduce regular sequences (Definition~\ref{regular sequence}) and associated Koszul resolutions for monoids in  the category $\mathcal{F}$ of functors over an essentially small linear symmetric monoidal category. In Section \ref{koszul}, given a monoid $A$ in $\mathcal{F}$, we define its associated Koszul complex and prove first that it is indeed a complex, and then that it is a resolution for a particular $A$-module, as in the classical case. The definition of the complex as well as the proof of the fact that it is a resolution are completely combinatorial in nature, thus avoiding the use of the tensor product of complexes. 

 In Section \ref{polynomials}, we define polynomials over the monoids in $\mathcal{F}$. In the classical framework of rings, the results we are interested in are proved for polynomial rings over fields so, in our context, we replace the fields by  {\em tensor idempotent commutative monoids} (see Definition~\ref{tensor idempotent}). With this, we compute the Hochschild cohomology functors and prove a relative analogue of Hilbert's syzygy theorem for polynomials over tensor idempotent commutative monoids in $\mathcal{F}$.\par

An important (and motivating) example where our results apply is the context of Green biset functors. In this case, the Burnside functor is commutative and tensor idempotent, so our results can be applied to polynomials over it.  

\section{Preliminaries}
\label{prems}

In what follows $R$ is a commutative ring with identity, denoted simply as $1$,  and $(\mathcal{X},\, \diamond,\, \uno,\,\alpha,\, \lambda,\, \rho,\, s)$ is an essentially small symmetric monoidal category enriched in $R$-Mod, in particular,  the functor $\,\_\,\diamond \,\_\, : \mathcal{X}\times \mathcal{X}\rightarrow \mathcal{X}$ is $R$-bilinear. The category of $R$-linear functors from $\mathcal{X}$ to $R$-Mod is denoted by $\mathcal{F}$, it is an abelian, symmetric monoidal, closed category with identity given by $I=\mathcal{X}(\uno,\,\_\,)$. The complete notation for $\mathcal{F}$ is $(\mathcal{F},\, \otimes,\, I,\, \alpha^{\mathcal{F}},\, \lambda^{\mathcal{F}},\, \rho^{\mathcal{F}},\, S)$.

\begin{ejem}
\label{ejemplos}
There are many examples of categories satisfying the conditions of $\mathcal{X}$, we mention two of them. 
\begin{itemize}
\item[i)]  The  \textit{biset category}, defined in \cite{biset}.
\item[ii)] The category of correspondences, defined in \cite{corr2}.
\end{itemize}
\end{ejem}

In the first case, the monoids in $\mathcal{F}$ are called Green biset functors and have been extensively studied in recent years. See, for example, \cite{lachica}. In the second case, monoids have been considered in Section 7 of~\cite{corfun-tensor}, under the name of {\em algebra functors}. It is shown there in particular that to any finite lattice $T$, one can associate a monoid $F_T$, and that conversely, any monoid satisfying some additional mild conditions is obtained in this way (Theorem 7.4 of~\cite{corfun-tensor}).\medskip\par
Recall that, given a monoid $A$ in $\mathcal{F}$, we denote its product $A\otimes A\rightarrow A$ as $\mu_A$ and its identity morphism $I\rightarrow A$ as $e_A$. For more notation and basic facts on monoids in~$\mathcal{F}$, and modules over them, we invite the reader to take a look at \cite{primero}.

\begin{defi}
Given a monoid $A$ in $\mathcal{F}$ and $C$ a subfunctor of $A$, we say that $C$ is a \textit{submonoid} of $A$ if $C$ is a monoid with the monoid structure of $A$. 
\end{defi}

The previous definition translates in the following way. Since $C$ is a subfunctor of $A$, we have a canonical morphism $i:C\rightarrow A$ in $\mathcal{F}$, so $C$ is a submonoid of $A$  if $\mu_A\circ(i\otimes i)$ has image in $C$ and $i\circ e_C=e_A$.

An important example of submonoid is the following.

\begin{defi}
Let $A$ be a monoid in $\mathcal{F}$. The commutant of $A$ at $x\in \mathcal{X}$ is
\begin{displaymath}
CA(x)=\{b\in A(x)\mid a\times b=A(s_{x\diamond y}^{y\diamond x})(b\times a), \forall a\in A(y), \forall y\in \mathcal{X}\},
\end{displaymath} 
where $s_{x\diamond y}^{y\diamond x}:x\diamond y\rightarrow y\diamond x$ is the symmetry in $\mathcal{X}$. It is a subfunctor of $A$ and, moreover, it is a \textit{commutative} submonoid of $A$.
\end{defi}

If $A$ and $D$ are monoids in $\mathcal{F}$ and $f:D\rightarrow A$ is a morphism of monoids, we have a \textit{restriction functor along $f$}, from $A$-Mod to $D$-Mod, which is an additive functor. When considering the arrow $e_{A}:I\rightarrow A$, the restriction from $A$-Mod to $\mathcal{F}$ is simply denoted as $\mathcal{R}_A$. Now, if $C$ is a submonoid of $A$, the restriction from $A$-Mod to $C$-Mod is denoted by $\mathcal{R}^A_C$. Since $\mathcal{R}^A_C$ is a forgetful functor, it is faithful and exact. This is the case of the functor $\mathcal{R}_A$ too.

Let $F$ be a functor in $\mathcal{F}$. Recall that, given a family $\mathcal{S}=(F_i)_{i\in I}$ of subfunctors of $F$, the intersection $H=\bigcap_{i\in I}F_i$ of $\mathcal{S}$  is the subfunctor of $F$ given by
\begin{displaymath}
H(x)=\bigcap_{i\in I}(F_i(x)),\quad H(\varphi)=F(\varphi)|_{H(x)}, 
\end{displaymath} 
for $x$ and $y$ objects and $\varphi:x\rightarrow y$ an arrow in $\mathcal{X}$.

\begin{rem}
Let $A$ be a monoid in $\mathcal{F}$ and $M$ be an $A$-module. If $\mathcal{S}$ is a family of submodules of $M$, then the intersection of $\mathcal{S}$ is a submodule of $M$.
\end{rem}

\begin{defi}
Let $F$ be a functor in $\mathcal{F}$. Given $\mathcal{G}$, a set of objects in $\mathcal{X}$, and for each $x\in \mathcal{G}$, a subset $\Gamma_x\subseteq F(x)$, the subfunctor generated by the data $(\mathcal{G}, \Gamma)$ is the intersection of the family
\begin{displaymath}
\mathcal{S}=\{F'\leq F\mid \Gamma_x\subseteq F'(x),\forall x\in \mathcal{G}\}.
\end{displaymath}
\end{defi}

\begin{defi}
Let $A$ be a monoid in $\mathcal{F}$ and $N$ be an $A$-module. Consider $\mathcal{G}$ and $\Gamma$ as before and
\begin{displaymath}
\mathcal{S}'=\{N'\leq N\mid \Gamma_x\subseteq N'(x),\forall x\in \mathcal{G}\}.
\end{displaymath}
We call the intersection of $\mathcal{S}'$ \textit{the submodule of $N$ generated by the data $(\mathcal{G}, \Gamma)$}.
\end{defi}

\begin{rem}
If $M$ is the submodule generated by $(\mathcal{G}, \Gamma)$, then $\Gamma_x\subseteq M(x)$ for all $x$ in~$\mathcal{G}$.
\end{rem}

\begin{lema}
\label{generado}
Let $A$ be a monoid in $\mathcal{F}$ and $\alpha_1,\ldots, \alpha_n$ be elements in $A(\uno)$. Let $M$ be the $A$-submodule of $A$ generated by $\alpha_1,\ldots, \alpha_n$. Then, for an object $x$ in $\mathcal{X}$,
\begin{displaymath}
M(x)=\{A(\rho_x)(a_1\times \alpha_1+\ldots +a_n\times \alpha_n)\mid a_i\in A(x)\}.
\end{displaymath}
\end{lema}
\begin{proof}
Let 
\begin{displaymath}
F_x=\{A(\rho_x)(a_1\times \alpha_1+\ldots +a_n\times \alpha_n)\mid a_i\in A(x)\}.
\end{displaymath}
Since, by the previous remark, $\alpha_1,\ldots, \alpha_n\in M(\uno)$ and $M$ is an $A$-submodule of $A$, we clearly have $F_x\subseteq M(x)$. 

Let us show that $F(x)=F_x$ defines an $A$-submodule of $A$ such that $\alpha_1,\ldots, \alpha_n\in F(\uno)$. With this we will have that $M$ is a submodule of $F$ and, hence, $M(x)\subseteq F_x$. 
It is easy to see that $F(x)$ is an $R$-submodule of $A(x)$ and clearly $\alpha_1,\ldots, \alpha_n\in F(\uno)$. Also, if $\varphi:x\rightarrow y$ is an arrow in $\mathcal{X}$,  we know that $A(\varphi\circ\rho_x)=A(\rho_x\circ(\varphi\diamond \uno))$ so, $A(\varphi)\left(A(\rho_x)(a_1\times \alpha_1+\ldots +a_n\times \alpha_n)\right)$ is equal to
\begin{displaymath}
A(\rho_x)\left(A(\varphi)(a_1)\times \alpha_1+\ldots +A(\varphi)(a_n)\times \alpha_n\right).
\end{displaymath}
Finally, for an object $y$ in $\mathcal{X}$, $m\in A(y)$ and any $n\in A(x\times \uno)$, we have
\begin{displaymath}
m\times A(\rho_x)(n)=A(y\diamond \rho_x)(m\times n)=A(\rho_{y\diamond x})(m\times n). 
\end{displaymath}
This shows that $F$ is an $A$-submodule of $A$.
\end{proof}

\begin{nota}
The functor $M$ of the previous lemma will be denoted by $A\langle \alpha_1,\ldots ,\alpha_n\rangle$, or simply as $A\langle\alpha\rangle$, where  $\alpha=(\alpha_1,\ldots, \alpha_n)$.
\end{nota}

\section{The Koszul complex}
\label{koszul}

In what follows $A$ is a monoid in $\mathcal{F}$. 

\begin{defi}
For an element $a\in A(\uno)$, we define  $L_a:A\rightarrow A$  as
\begin{displaymath}
L_{a, x}:A(x)\rightarrow A(x)\quad b\mapsto A(\lambda_x)(a\times b),
\end{displaymath}
for $x$ an object in $\mathcal{X}$. For convenience, we will often denote $L_a$ simply as $a\mult$.
\end{defi}
This defines a natural transformation in $A$. Furthermore, if $a\in CA(\uno)$, then $L_a$ is clearly a morphism of $A$-modules.

\begin{nota}
Given a natural number $n>0$, we denote by $A^n$ the coproduct $\oplus A$ of $A$ with itself $n$ times. When considering the binomial coefficient, $\binom{n}{p}$, for $0\leq p\leq n$, we will write $K_p^n(A)$ for
\begin{displaymath}
A^{\binom{n}{p}}=\bigoplus_{S\subseteq T\, |S|=p} A, 
\end{displaymath}
where $T=\{1,\ldots, n\}$. When necessary, we will write $A_S$ for the summand of $K_p^n(A)$ corresponding to the set $S$.
\end{nota}

To define the Koszul complex, we fix a natural number $n>1$ and consider a family of elements $\alpha_1,\ldots ,\alpha_n$ in $CA(\uno)$. We fix these elements and abbreviate them as $\alpha=(\alpha_1,\ldots, \alpha_n)$. 

We begin by defining, for $p\leq n$, the arrow $d:K_p^n(A)\rightarrow K_{p-1}^n(A)$,
\begin{displaymath}
d: \bigoplus_{S\subseteq T\, |S|=p} A\longrightarrow \bigoplus_{S\subseteq T\, |S|=p-1} A.
\end{displaymath}
Given $S=\{i_1,\ldots ,i_p\}\subseteq T$, the arrow $d_S:A_S\rightarrow K_{p-1}^n(A)$ is given by
\begin{displaymath}
d_S=\sum_{k=1}^p(-1)^{k+1}(\alpha_{i_k}\mult)_{\hat{i_k}},
\end{displaymath}
where the {subscript} $\hat{i_k}$ indicates that $\alpha_{i_k}\mult$ has as codomain $A_{S\setminus\{i_k\}}$. In the case $p=1$, the arrow $d:A^n\rightarrow A$ is equal to $\oplus_{i=1}^n d_{\{i\}}$ and each $d_{\{i\}}$ is just $L_{\alpha_i}:A\rightarrow A$, so we denote $d$ as $L_{\alpha}$.

\begin{defi}
For $A$ and $\alpha$ as before, we define $K_A(\alpha)$ as the sequence in $\mathcal{F}$
\[
\xymatrix{
0\ar[r] &K_n^n(A)\ar[r]^-d&\ldots\,K_p^n(A)\ar[r]^-d&K_{p-1}^n(A)\ar[r]^-d&\ldots\, A^n\ar[r]^-{L_{\alpha}}&A\ar[r]&0,
}
\]
with $A$ in degree 0. As we will see next, this sequence is a complex of $A$-modules. We call it the Koszul complex of $A$ for $\alpha$.
\end{defi}

Even though some of the properties of $K_A(\alpha)$ follow in a similar way as in the classical case, we prove them for the sake of completeness.

\begin{lema}
\label{escomple}
With $A$ and $\alpha$ as before, $K_A(\alpha)$ is a complex of $A$-modules.
\end{lema}
\begin{proof}
Since $\alpha_1,\ldots ,\alpha_n$ are in $CA(\uno)$, the sequence $K_A(\alpha)$ is clearly in the category of $A$-modules.

Next, we prove $d\circ d=0$. 

Let us fix $S=\{i_1,\ldots, i_p\}$. If we start by $d_S=\sum_{k=1}^p(-1)^{k+1}(\alpha_{i_k}\mult)_{\hat{i_k}}$, then in $d\circ d$, for each $k$  we need only to consider the arrow $d_{S\setminus\{i_k\}}$. Now,
\begin{displaymath}
d_{S\setminus\{i_k\}}=\sum_{j=1}^{k-1}(-1)^{j+1}(\alpha_{i_j}\mult)_{\hat{i_k},\hat{i_j}}+\sum_{j=k+1}^{p}(-1)^j(\alpha_{i_j}\mult)_{\hat{i_k},\hat{i_j}},
\end{displaymath}
where $\hat{i_k},\hat{i_j}$ indicates that $\alpha_{i_j}\mult$ has codomain $A_{S\setminus\{i_k, i_j\}}$. Hence, the composition $d\circ d_S$ can be written as
\begin{displaymath}
\sum_{k=1}^p\left(\sum_{j=1}^{k-1}(-1)^{k+j}((\alpha_{i_j}\times\alpha_{i_k})\mult)_{\hat{i_j},\hat{i_k}}+\sum_{j=k+1}^{p}(-1)^{k+j+1}((\alpha_{i_j}\times\alpha_{i_k})\mult)_{\hat{i_j},\hat{i_k}}\right),
\end{displaymath}
since clearly $(\alpha_{i_j}\mult)_{\hat{i_j},\hat{i_k}}\circ (\alpha_{i_k}\mult)_{\hat{i_k}}$ is equal to $((\alpha_{i_j}\times\alpha_{i_k})\mult)_{\hat{i_j},\hat{i_k}}$. 

Finally, let $e$ and $f$ be in $\{1,\ldots, p\}$ and, without loss of generality, suppose $e<f$.  Then, in the sum above, for $k=e$ and $j=f$, we find $(-1)^{e+f+1}((\alpha_{i_f}\times\alpha_{i_e})\mult)_{\hat{i_f},\hat{i_e}}$, and for $k=f$ and $j=e$, we find $(-1)^{f+e}((\alpha_{i_e}\times\alpha_{i_f})\mult)_{\hat{i_e},\hat{i_f}}$. These are the only two times we find the arrow $((\alpha_{i_e}\times\alpha_{i_f})\mult)_{\hat{i_e},\hat{i_f}}$ in the sum above, and all the summands are of this form,  hence $d\circ d_S=0$ and $d\circ d=0$.
\end{proof}

\subsection*{The Koszul resolution}

For this section we follow the lines of Matsumura \cite{matsu}.

\begin{defi}\label{regular sequence}
For an $A$-module $M$, an element $a\in CA(\uno)$ is called $M$-regular if, for any object $x$ in $\mathcal{X}$ and any $m\in M(x)$,
\begin{displaymath}
m\neq 0\Rightarrow a\times m\neq 0.
\end{displaymath}
{The (ordered) sequence $\alpha=(\alpha_1,\ldots, \alpha_n)$, with $\alpha_i\in CA(\uno)$, is called {\em regular} if}
\begin{enumerate}
\item $\alpha_1$ is $A$-regular, $\alpha_2$ is $\left(A/(A\langle\alpha_1\rangle)\right)$-regular, ..., $\alpha_n$ is $(A/A\langle\alpha_1,\ldots, \alpha_{n-1}\rangle))$-regular.
\item $A/(A\langle\alpha\rangle)\neq 0$.
\end{enumerate}
\end{defi}

%\begin{rem}
%In what follows, the family $\alpha=(\alpha_1,\ldots, \alpha_n)$, with $\alpha_i\in CA(\uno)$, should be considered as an ordered sequence.
%\end{rem}

\def\ideal{I_\alpha}
\begin{teo}
\label{esreso}
If $\alpha=(\alpha_1,\ldots, \alpha_n)$, with $\alpha_i\in CA(\uno)$, is a regular sequence for $A$, then $K_A(\alpha)$ is a resolution for $A/\ideal$, with $\ideal=A\langle\alpha\rangle$. That is, the following sequence is an exact sequence of $A$-modules, 
\[
\xymatrix{
0\ar[r] &K_n^n(A)\ar[r]^-d&\ldots\,\ar[r]^-d&K_{2}^n(A)\ar[r]^-d& {K_1^n(A)}\ar[r]^-{L_{\alpha}}&A\ar[r]&A/\ideal\ar[r]&0.
}
\]
\end{teo}
\begin{proof}
For this proof we will use the notation $\alpha^i=(\alpha_1,\ldots , \alpha_i)$, for $1<i<n$.

We begin by noticing that, for $1\leq p\leq n$, we have the following decomposition of $A$-modules, given by Pascal's rule,
\begin{equation}\label{decomposition}
K_p^n(A)=K_p^{n-1}(A)\oplus K_{p-1}^{n-1}(A).
\end{equation}
 We denote by $\iota$ the embedding  $K_p^{n-1}(A)\hookrightarrow K_p^n(A)$, and by $\tau$ the projection  $K_p^n(A)\twoheadrightarrow K_{p-1}^{n-1}(A)$.\par
Let us see what happens with {the decomposition~(\ref{decomposition})} in the Koszul complex. We write $T=T'\cup \{n\}$ with $T'=\{1,\ldots , n-1\}$. Then, in $K_p^n(A)$, we can identify the two summands in the following way
\begin{displaymath}
K_p^{n-1}(A)=\bigoplus_{S'\subseteq T'\, |S'|=p}\!\!A\quad \textrm{and}\quad K_{p-1}^{n-1}(A)=\bigoplus_{\substack{S=S'\cup\{n\}\\S'\subseteq T'\, |S'|=p-1}}\!\!A.
\end{displaymath}
The arrow $d:K_p^n(A)\rightarrow K_{p-1}^n(A)$ remains the same when restricted to $K_p^{n-1}(A)$, that is, $d|_{K_p^{n-1}(A)}$ is actually
\begin{displaymath}
d:K_p^{n-1}(A)\rightarrow K_{p-1}^{n-1}(A).
\end{displaymath}
Now we suppose $p>1$ and restrict $d$ to $K_{p-1}^{n-1}(A)$. In $A_S$, for a set $S=S'\cup\{n\}$ with $S'=\{i_1,\ldots ,i_{p-1}\}$, the arrow $d_S$ is equal to
\begin{displaymath}
\sum_{k=1}^{p-1}(-1)^{k+1}(\alpha_{i_k}\mult)_{\hat{i_k}}+(-1)^{p+1}(\alpha_n\mult)_{\hat{n}}.
\end{displaymath}
This can be written as $d_{S'}+(-1)^{p+1}L_{\alpha_n}$, taking into account that in $d_{S'}$ the subindex $\hat{i_k}$ means that the codomain is $A_{(S'\setminus\{i_k\})\cup \{n\}}$ and $L_{\alpha_n}:A_S\rightarrow A_{S'}$. That is, $d|_{K_{p-1}^{n-1}(A)}$ is given by $d+(-1)^{p-1}L_{\alpha_n}$, with 
\begin{displaymath}
d:K_{p-1}^{n-1}(A)\rightarrow K_{p-2}^{n-1}(A)\quad \textrm{and with}\quad L_{\alpha_n}:K_{p-1}^{n-1}(A)\rightarrow K_{p-1}^{n-1}(A)
\end{displaymath}
being the product by $\alpha_n$ in each summand. With this, the  (split) short exact sequence of $A$-modules
\[
\xymatrix{
0\ar[r]&K_p^{n-1}(A)\ar[r]^-{\iota}&K_p^n(A)\ar[r]^-{\tau}&K_{p-1}^{n-1}(A)\ar[r]&0,
}
\]
induces a short exact sequence of complexes 
\[
\xymatrix{
0\ar[r]&K_A(\alpha^{n-1})\ar[r]&K_A(\alpha)\ar[r]&K^{[1]}_A(\alpha^{n-1})\ar[r]&0,
}
\]
where $K^{[1]}_A(\alpha^{n-1})$ is the complex obtained by shifting the degrees of $K_A(\alpha^{n-1})$ up by 1 and the differentials are the same. 
Indeed, for $0\leq p\leq n$, consider
\[
\xymatrix{
0\ar[r]&K_p^{n-1}(A)\ar[d]^-d\ar[r]^-{\iota}&K_p^n(A)\ar[d]^-d\ar[r]^-{\tau}&K_{p-1}^{n-1}(A)\ar[d]^-d\ar[r]&0\\
0\ar[r]&K_{p-1}^{n-1}(A)\ar[r]^-{\iota}&K_{p-1}^n(A)\ar[r]^-{\tau}&K_{p-2}^{n-1}(A)\ar[r]&0.
}
\]
The cases $p=0$ and 1 are straightforward. Suppose $p>1$. By the description of $d|_{K_p^{n-1}(A)}$, above, the square on the left commutes. For the square on the right, since 
$L_{\alpha_n}$ has image in $K_{p-1}^{n-1}(A)$, then clearly $\tau d=d\tau$. 

Since $\mathcal{F}$ is an abelian category, we have the following long exact sequence of homology 
%\[\xymatrix{
%\ldots\ar[r]&H_{q+1}(K^1_A(\alpha^{n-1}))\ar[r]^-{\delta}&H_q(K_A(\alpha^{n-1}))\ar[r]& H_q(K_A(\alpha))\ar[r]& \\
%&H_q(K^1_A(\alpha^{n-1}))\ar[r]^-{\delta}&\ldots\quad\quad\quad\quad\quad&&
%}
%\]
\[\xymatrix{
\ldots\ar[r]&H_{p+1}(K^{[1]}_A(\alpha^{n-1}))\ar[r]^-{\delta}&H_p(K_A(\alpha^{n-1}))\ar[r]& H_p(K_A(\alpha))\ar[r]& 
}
\]
\[
\xymatrix{
&H_p(K^{[1]}_A(\alpha^{n-1}))\ar[r]^-{\delta}&\ldots\quad\quad\quad\quad\quad&&
}
\]
Notice that $H_p(K^{[1]}_A(\alpha^{n-1}))=H_{p-1}(K_A(\alpha^{n-1}))$, so 
\begin{displaymath}
\delta:H_{p-1}(K_A(\alpha^{n-1}))\rightarrow H_{p-1}(K_A(\alpha^{n-1})).
\end{displaymath}
Now, $H_{p-1}(K_A(\alpha^{n-1}))$ is a subquotient of $K_{p-1}^{n-1}(A)$ and (since in $\mathcal{F}$ all limits and colimits are defined pointwise) the connecting morphism $\delta$ is given, in an object $x$, through  the composition $(\iota_x)^{-1}d_x(\tau_x)^{-1}$ (a short for taking inverse image by $\tau_x$, then image by $d_x$, and then taking inverse image by $i_x$). So, by the description of $d|_{K_{p-1}^{n-1}(A)}$, given before, we have $\delta=(-1)^{p-1}L_{\alpha_n}$.

To finish the proof, we observe first that, by Lemma \ref{generado}, Im$L_{\alpha}=\ideal$. Next we proceed, as in the classical case, by induction on $n$.

For $n=1$, the sequence of $A$-modules
\[
\xymatrix{
0\ar[r]&A\ar[r]^-{L_{\alpha_1}}&A\ar[r]&A/I_{\alpha_1}\ar[r]&0
}
\]
is exact because $\alpha_1$ is $A$-regular.

Now suppose the theorem to be true for $n-1$. Then, by induction hypothesis, for $p>1$, in the previous long exact sequence of homology we have
\[\xymatrix{
0=H_p(K_A(\alpha^{n-1}))\ar[r]&H_p(K_A(\alpha))\ar[r]& H_{p-1}(K_A(\alpha^{n-1}))=0. 
}
\]
So $H_p(K_A(\alpha))=0$ for $p>1$. Finally, for $p=1$ we have
\[\xymatrix{
0\ar[r]&H_1(K_A(\alpha))\ar[r]&H_0(K_A(\alpha^{n-1}))\ar[r]^-{\delta}&H_0(K_A(\alpha^{n-1}))\ar[r]&\ldots 
}
\]
But $H_0(K_A(\alpha^{n-1}))=A/A\langle\alpha^{n-1}\rangle$ and $\delta=L_{\alpha_n}$. Since $\alpha_n$ is $(A/A\langle\alpha^{n-1}\rangle)$-regular, we must have $H_1(K_A(\alpha))=0$.
\end{proof}

%\begin{coro}
%Let
%\[
%\xymatrix{
%0\ar[r] &\Lambda^n(A^n)\ar[r]^-d&\ldots\,\ar[r]^-d&\Lambda^{2}(A^n)\ar[r]^-d& A^n\ar[r]^-{L_{\alpha}}&A\ar[r]&A/I\ar[r]&0.
%}
%\]
%be as in the previous theorem but suppose also that... FISH
%\end{coro}

\section{Functors of polynomials}\label{polynomials}

\subsection*{General properties}

We begin by recalling the following standard definition and notation. %from Artin \cite{artin}.

\begin{defi}
Let  $t_1,\ldots ,t_n$ be variables. Given a monomial $t_1^{i_1}\ldots t_n^{i_n}$, where the exponents $i_{\nu}$ are nonnegative integers, the vector $i=(i_1,\ldots ,i_n)$ of exponents is called a \textit{multi-index} {\em of size $n$}. We write the monomial symbolically as 
\begin{displaymath}
t^i= t_1^{i_1}\ldots t_n^{i_n}.
\end{displaymath}
For the multi-index $(0,\ldots ,0)$, we fix $t^0=t_1^0\ldots t_n^0=1$. 
\end{defi}

Given an $R$-module $U$, we can consider the $R$-module $U[t_1,\ldots ,t_n]$. With the previous definition, a polynomial  $f(t)=f(t_1,\ldots ,t_n)$ in $U[t_1,\ldots ,t_n]$, can be written in exactly one way in the form
\begin{displaymath}
f(t)=\sum_i a_it^i,
\end{displaymath} 
where $i$ runs through all the multi-indices $(i_1,\ldots ,i_n)$, the coefficients $a_i$ are in $U$ and only finitely many of theme are different from zero. Finally, $U[t_1,\ldots ,t_n]$ contains $U$ as the submodule of constant polynomials, that is, polynomials with all coefficients equal to zero except, possibly, the one corresponding to the multi-index $(0, \ldots ,0)$.

We keep the variable $t$ whether we have one or several variables. That is, if we have only one variable, we continue to write $t^i$, instead of $t_1^{i_1}$.

Given an object $F$ in $\mathcal{F}$, we define \textit{the functor of polynomials in $F$}, denoted by $F[t_1,\ldots ,t_n]$, or simply by $F_n$, as follows. In an object $x$ in $\mathcal{X}$ it is defined as the $R$-module $F_n(x):=F(x)[t_1,\ldots ,t_n]$ and, in an arrow $\varphi:x\rightarrow y$ in $\mathcal{X}$, as the  $R$-linear map
\begin{displaymath}
F_n(\varphi):F_n(x)\rightarrow F_n(y),\quad p(t)=\sum_{i}a_it^i\mapsto F(\varphi)(p(t)):=\sum_{i}F(\varphi)(a_i)t^i,
\end{displaymath}
for $p(t)\in F_n(x)$.

\begin{prop}
\label{Atismonoid}
Let  $t_1,\ldots ,t_n$ be variables and $F$ be an object in $\mathcal{F}$. Then $F_n$ is an object in $\mathcal{F}$. Also, if $A$ is a monoid in $\mathcal{F}$, then  $A_n$ inherits a monoid structure from $A$, 
\begin{displaymath}
A_n(x)\times A_n(y)\rightarrow A_n(x\diamond y),
\end{displaymath}
by sending $p(t)=\sum_{i}a_it^i\in A_n(x)$ and $q(t)=\sum_jb_jt^j\in A_n(y)$ to 
\begin{displaymath}
(p\times q)(t):=\sum_{i,\, j}(a_i\times b_j)t^{i+j}\in A_n(x\diamond y).
\end{displaymath}
The identity element is given by the identity element of $A$, $\varepsilon(t):=\varepsilon$,  seen as a constant polynomial in $A_n(\uno)$. %Finally, the previous structure, restricted to $A_0=A$ on the left and, in turn, on the right, yields a structure of $A$-bimodule for $A_m$. 
\end{prop}
\begin{proof}
The first assertion is clear, since $F_n(\varphi)$ extends $F(\varphi)$ and $F$ is a functor. The rest of the proposition also follows easily, since the product defined in $A_n$ extends the one in $A$. We just notice, as usual, that $(p\times q)(t)=\sum_kp_kt^k$, where
\begin{displaymath}
p_k=\sum_{i+j=k}a_i\times b_j
\end{displaymath}
and $k$ is a multi-index.
\end{proof}

%FISH: esto define un endofuntor en $\mathcal{F}$ y puede verse como un monoide $\mathbb{N}^{\times n}$-graduado? :-[

\begin{rem}
\label{propiedades}
Let  $t_1,\ldots ,t_n$ be variables, $F$ an object in $\mathcal{F}$ and $A$ a monoid in $\mathcal{F}$. From the previous proposition, we have the following observations:
\begin{enumerate}
\item Sending $F$ to $F_n$ defines an endofunctor of the category $\mathcal{F}$. 
\item The functor $F_n$ is isomorphic to a coproduct, running over all the multi-indices, of copies of $F$. In the case of $A$, this is an isomorphism of $(A,\,A)$-bimodules.
\item Let $u$ and $v$ be variables. Then $F$ is a subfunctor of $F[u]$. More generally,  $F[u]$ is a subfunctor of $F[u,v]$ and, in fact, $F[u,v]=(F[u])[v]$.  On the other hand, $A$ is a submonoid of $A[u]$ and  $A[u]$ is a submonoid of $A[u,v]$. In particular, any $A[u]$-module is an $A$-module. This observation extends in an obvious way to $n$ variables.
\item If $A$ is a commutative monoid, then $A_n$ is commutative too. 
\item Each variable $t_i$ can be found in $A_n(\uno)$ as $\varepsilon t_i$. Moreover $\lae t_i\in CA_n(\uno)$.
\end{enumerate}
\end{rem}

\begin{lema} \label{C[t] tensor D}Let $C$ and $D$ be monoids in $\mathcal{F}$.
\begin{enumerate}
\item Let $t$ be a variable. Then there are isomorphisms of monoids 
$$C[t]\otimes D\cong (C\otimes D)[t]\cong C\otimes D[t].$$
\item Let $n,m\in\N$, and $u_1,\ldots,u_n,v_1,\ldots v_m$ be variables. Then there is an isomorphism of monoids
$$C[u_1,\ldots,u_n]\otimes D[v_1,\ldots,v_m]\cong (C\otimes D)[u_1,\ldots,u_n,v_1,\ldots v_m].$$
\end{enumerate}
\end{lema}
\begin{proof} 1. Since the tensor product of monoids is commutative up to isomorphism, it suffices to prove the first isomorphism $C[t]\otimes D\cong (C\otimes D)[t]$. Now $C[t]=\mathop{\bigoplus}_{n\in \N}\limits Ct^n$, and $Ct^n\cong C$ in $\CF$. It follows that we have isomorphisms in $\CF$
$$C[t]\otimes D=\big(\bigoplus_{n\in\N}Ct^n\big)\otimes D=\bigoplus_{n\in\N}(Ct^n\otimes D)\cong \bigoplus_{n\in\N}(C\otimes D)\cong (C\otimes D)[t],$$
where the composed isomorphism $\Phi:C[t]\otimes D\to (C\otimes D)[t]$ is given by the identification $\Phi_n$ of $Ct^n\otimes D$ with $(C\otimes D)t^n$, for all $n\in \N$. In particular, the restriction $\Phi_0$ of $\Phi$ to $C\otimes D=Ct^0\otimes D$ is the identity morphism $C\otimes D=Ct^0\otimes D\to (C\otimes D)t^0=C\otimes D$, so $\Phi$ maps the identity element of $C[t]\otimes D$ to the identity element of $(C\otimes D)[t]$.\par
Hence all we have to check is that $\Phi$ is compatible with the product of monoids, which follows from the following commutative diagram, for all $n,m\in\N$,
\[\xymatrix@C=1ex{
C[t]\otimes D&\otimes& C[t]\otimes D\ar[rrr]^-{\Phi\otimes\Phi}&&&(C\otimes D)[t]&\otimes& (C\otimes D)[t]\\
Ct^n\otimes D\ar[drr]\ar@{^{(}->}[u]&\otimes& Ct^m\otimes D\ar[dll]\ar@{^{(}->}[u]\ar[rrr]^-{\Phi_n\otimes\Phi_m}&&&(C\otimes D)t^n\ar@{^{(}->}[u]&\otimes\ar@{=}[d]& (C\otimes D)t^m\ar@{^{(}->}[u]\\
Ct^n\otimes Ct^m\ar[d]_{\mu_{C[t]}}&\otimes& D\otimes D\ar[d]_{\mu_{D}}&&&(C\otimes D)t^n&\otimes\ar[d]_{\mu_{(C\otimes D)[t]}}
& (C\otimes D)t^m\\
Ct^{n+m}&\otimes& D\ar[rrr]^-{\Phi_{n+m}}&&&&\makebox[4ex]{$(C\otimes D)t^{n+m}.$}&\\
}
\]
2. This follows from Assertion 1, applying $n$ times the left hand side isomorphism and $m$ times the right hand side isomorphism.
\end{proof}

The following notion already appears (without a name) in Proposition~5.8 of~\cite{primero}:

\begin{defi} \label{tensor idempotent}Let $A$ be a monoid in $\CF$. We say that $A$ is {\em tensor idempotent} if $\mu_A:A\otimes A\to A$ is an isomorphism in $\CF$, or equivalently, an isomorphism of $(A,A)$-bimodules.
\end{defi}

\begin{ejem}
Clearly, the identity functor $I$ is tensor idempotent and, by Proposition 5.8 in \cite{primero}, any quotient of it in $\mathcal{F}$ is also tensor idempotent.  So, if $\mathcal{F}$ is the category of \textit{biset functors} (see Example \ref{ejemplos}), then the Burnside functor $RB$  is tensor idempotent. Moreover, since any biset subfunctor of $RB$ is an ideal, any quotient biset functor of $RB$ is a tensor idempotent monoid. For example, the simple biset functor $S_{1,k}$, where $k$ is any field, and the functor of rational representations $kR_{\rac}$, where $k$ is a field of characteristic~0 (see \cite{biset}), are tensor idempotent.
\end{ejem}

%FISH para S: more examples of tensor idempotent monoids?}

\begin{prop} \label{A tensor A iso A}
Suppose that $A$ is a  tensor idempotent commutative monoid in $\mathcal{F}$. Then:
\begin{enumerate}
\item $A_n\otimes A_n\cong A_{2n}$ as monoids in $\mathcal{F}$.  
\item Also, via this isomorphism, if we let $A_{2n}=A[u_1,\ldots ,u_n,v_1,\ldots ,v_n]$, then the kernel of $\mu_{A_n}:A_n\otimes A_n\rightarrow A_n$
is equal to the ideal $A_{2n}\langle u_1-v_1,\ldots ,u_n-v_n\rangle$.
\end{enumerate}
\end{prop}
\begin{proof}1. Taking $C=D=A$ and $m=n$ in Assertion 2 of Lemma~\ref{C[t] tensor D}, we get that $A_n\otimes A_n\cong (A\otimes A)_{2n}$ as monoids in $\CF$. Now since $A$ is commutative, the morphism $\mu_n:A_n\otimes A_n\to A_n$ is a morphism of monoids. Then the assumption implies that $A\otimes A\cong A$ as monoids, so $(A\otimes A)_{2n}\cong A_{2n}$ as monoids, and $A_n\otimes A_n\cong A_{2n}$ as claimed.\mpn
2. Let $u=(u_1,\ldots,u_n)$, $v=(v_1,\ldots,v_n)$ and $t=(t_1,\ldots ,t_n)$. Then
$$A[u]=\bigoplus_{i}Au^i,\;A[v]=\bigoplus_jAv^j,\;\hbox{and}\;A[t]=\bigoplus_kAt^k,$$
where $i,j,k$ run through multi-indices of size $n$. The morphism
$$\mu_{A_n}:A[u]\otimes A[v]\to A[t]$$
is induced by the morphisms $\mu_{i,j}:Au^i\otimes Av^j\cong A\otimes A\stackrel{\mu_A}{\to}A\cong At^{i+j}$. We have a commutative diagram
\[\xymatrix@R=2ex@C=2ex{
A_n\ar@{=}[d]& &A_n\ar@{=}[d]&&A_n\otimes A_n\ar@{=}[d]\\
A[u]\ar@{=}[dd]&\otimes& A[v]\ar@{=}[dd]\ar[rr]^-\cong&&(A\otimes A)[u,v]\ar@{=}[dd]\\
&&&\\
\bigoplus_{i}\limits Au^i&\otimes\ar[dd]_-{\bigoplus_{i,j}\limits \mu_{i,j}} &\bigoplus_j\limits Av^j\ar[rr]^-\cong&&\bigoplus_{i,j}\limits (A\otimes A)u^iv^j\ar[dd]_-{\bigoplus_{i,j}\limits\mu_A}^-{\cong}\\
&&&\\
&\bigoplus_{k}\limits At^{k}\ar@{=}[d]&&&\bigoplus_{i,j}\limits Au^iv^j\ar[lll]_-{\pi}\ar@{=}[d]\\
&A_n&&&A_{2n}
}
\]
where the bottom right vertical isomorphism is induced by $\mu_A:A\otimes A\to A$, and the bottom horizontal morphism $\pi:A_{2n}\to A_n$ is induced by the morphisms $Au^iv^j\to At^{i+j}$ which are the identity on $A$, and map $u^iv^j$ to $t^{i+j}$. It follows that for all $l\in\{1,\ldots,n\}$, the elements $u_l$ and $v_l$ of $A_{2n}(\uno)$ are both mapped to $t$ by $\pi$. Hence $\Ker\,\pi$ contains the ideal $J=A_{2n}\langle u_1-v_1,\ldots,u_n-v_n\rangle$ of $A_{2n}$ generated by $u_1-v_1,\ldots,u_n-v_n$.\par
Now $A[u]$ is a submonoid of $A[u,v]$, and we clearly have $A[u,v]=A[u]+J$. Thus
$J\leq \Ker\,\pi\leq A[u]+J$, so $\Ker\,\pi=J+\big(\Ker\,\pi\cap A[u]\big)$. But the restriction of $\pi$ to $A[u]$ is an isomorphism, as it is induced by the morphisms $Au^i\to At^i$ which are the identity on $A$ and map $u^i$ to $t^i$, for all $i\in\N$. Then $\Ker\,\pi\cap A[u]=0$, and $\Ker\,\pi=J$, as was to be shown.
\end{proof}

\subsection*{Hochschild cohomology and Hilbert syzygy theorem}

In what follows, we refer to Section 4 of \cite{resmac} for the notion of projectivity {\em relative to a functor $\CR:\CC\to \CD$ between categories  $\CC$ and $\CD$}. Suppose that $\CC$ and $\CD$ are abelian categories and that $\CR$ is a faithful exact additive functor admitting a left adjoint. We will say that an object $M$ of $\CC$ is {\em $\CD$-projective} if it is projective relative to $\CR$ in the sense of~\cite{resmac}. Similarly, we will say that a complex $\CL$ in $\CC$ is {\rm $\CD$-split} if $\CR(\CL)$ is a split complex in $\CD$. \par
Now for an object $M$ of $\CC$, we say that a complex 
\begin{equation}
\ldots \to L_i\to L_{i-1}\to \ldots L_0\to M\to 0\tag{$\CL$}
\end{equation}
in $\CC$ is a {\em $\CD$-split resolution of $M$} if it is exact in $\CC$ and if $\CR(\CL)$ is a split complex in~$\CD$. Since $\CR$ is exact, this is equivalent to saying that $\CL$ is exact and $\CR(\CL)$ is {\em split exact} in~$\CD$. With this terminology, Lemma~4.6 of~\cite{resmac} says that every object $M$ of $\CC$ admits a $\CD$-split resolution $\CL$, where the objects $L_i$, for $i\geq 0$, are $\CD$-projective, and that such a resolution is unique up to homotopy.\medskip\par
In our context, we consider a monoid $A$ in $\CF$, and the functor $\CR$ will be the restriction functor $\CR_A:A\hbox{-Mod}\to\CF$. Then, every $A$-module $M$ admits an $\CF$-split resolution by $\CF$-projective objects, and such a resolution is unique up to homotopy. 
%We recall that, by Lemma~4.6 of~\cite{resmac}, since $\mathcal{R}_A$ is faithful and admits a left adjoint, every $A$-module $M$ admits an $\mathcal{F}$-split resolution, {\color{red}$\mathcal{L}$. That is, $\mathcal{L}$ is a resolution 
\medskip\par

From now on, we suppose that $A$ is a tensor idempotent commutative monoid in~$\CF$, as in Proposition~\ref{A tensor A iso A}. In particular, $A_n$-bimodules coincide with $A_{2n}$-modules.

\begin{teo}\label{koszul An}
Let $A$ be a tensor idempotent commutative monoid in $\mathcal{F}$. Let $C=A_{2n}=A[u_1,\ldots ,u_n,v_1,\ldots ,v_n]$ and $\alpha_i=u_i-v_i$ for $i=1,\ldots ,n$. Also, let $M$ be an $A_n$-bimodule and, for a positive integer $p$, consider $\mathcal{H}H^p(A_n,\, M)$, the Hochschild cohomology of $A_n$ with coefficients in $M$ (see \cite{primero}).  Then:
\begin{enumerate}
\item The sequence $\alpha=(\alpha_1,\ldots ,\alpha_n)$ is a regular sequence for $C$ and the Koszul resolution of $C/(C\langle \alpha\rangle)\cong A_n$,
\[
\xymatrix@C=4ex{
0\ar[r] &K_n^n(C)\ar[r]^-d&\ldots\,\ar[r]^-d&K_{2}^n(C)\ar[r]^-d& K_1^n(C)\ar[r]^-{L_{\alpha}}&C\ar[r]^-{\mu_{A_n}}&A_n\ar[r]&0
}\tag{$K_C(\alpha)_{\mu}$}
\]
is an $\CF$-split resolution of $A_n$ by $\CF$-projective $A_{2n}$-modules. 
\item  $\mathcal{H}H^p(A_n,\, M)=0$ for $p>n$ and $\mathcal{H}H^p(A_n,\, A_n)\, \cong\, K_p^n(A_n)$ for $p=0,\ldots ,n$.
\end{enumerate}
\end{teo}
\begin{proof}
1. Let $C'=A[u_1,\ldots , u_n]$. We notice that, by Remark \ref{propiedades}, $C=C'[v_1,\ldots ,v_n]$ and also $C=C'[\alpha_1,\ldots , \alpha_n]$. The last equality implies that $\alpha$ is a regular sequence for $C$ and hence, by Theorem \ref{esreso}, we know that $K_C(\alpha)_\mu$ is an exact sequence of $C$-modules. We must show now that $C$ and the $K_i^n(C)$, with $1\leq i\leq n$ are $\CF$-projective and that $\mathcal{R}_C\big(K_C(\alpha)_\mu\big)$ is  a split complex. Since $C$ and the $K_i^n(C)$, with $1\leq i\leq n$, are projective $C$-modules (each is a coproduct of copies of $C$), then, by Remark 4.5 in \cite{resmac}, they are projective with respect to $\mathcal{R}_C$, i.e. $\CF$-projective. 

To prove that  
\[
\xymatrix{
0\ar[r] &\mathcal{R}_C\big(K_n^n(C)\big)\ar[r]^-d&\ldots\,\ar[r]^-d& \mathcal{R}_C\big(K_1^n(C)\big)\ar[r]^-{L_{\alpha}}&\mathcal{R}_C(C)\ar[r]^-{\mu_{A_n}}&\mathcal{R}_C(A_n)\ar[r]&0
}
\]
is  split, we consider first the restriction to $C'$. That is, each of the objects  $\CR_{C'}^C(A_{n})$, $\CR_{C'}^C(C)$ and $\CR_{C'}^C\big(K_i^n(C)\big)$ with $1\leq i\leq n$ is isomorphic to a coproduct of copies of $C'$, by Remark \ref{propiedades}. Hence, they are all projective $C'$-modules and   $\CR_{C'}^C\big(K_C(\alpha)_\mu\big)$ is  a split complex of $C'$-modules.  Hence, $\mathcal{R}_C(K_C(\alpha)_\mu)=\mathcal{R}_{C'}(\mathcal{R}^{C}_{C'}(K_C(\alpha)_\mu))$ is  a split complex.\medskip\par

2. By  Lemma 5.1 and Remark 5.2  in \cite{primero}, the bar resolution and the Koszul resolution of $A_n$ are homotopy equivalent. Hence, we can calculate the Hochschild cohomology of $A_n$ using $K_C(\alpha)_{\mu}$. Then, clearly $\mathcal{H}H^p(A_n,\, M)=0$.

Again using the Koszul resolution, in the Hochschild cochain complex of $A_n$, for $1\leq p\leq n$, we have the morphism
\begin{displaymath}
\Phi :\mathcal{H}_{C}(K_p^n(C),\, A_n)\longrightarrow \mathcal{H}_C(K_{p+1}^n(C),\, A_n),
\end{displaymath}
which, in an object $x$ of $\mathcal{X}$, sends an arrow $t\in\mathrm{Hom}_{C\hbox{-}\mathrm{Mod}}(K_p^n(C),\, (A_n)_x)$ to $t\circ d\in \mathrm{Hom}_{C\hbox{-}\mathrm{Mod}}(K_{p+1}^n(C),\, (A_n)_x)$. Now, from Section 3, we know that $d$ is the sum of some $L_{\alpha_{i_k}}:C\rightarrow C$, with $L_{\alpha_{i_k}}=\alpha_{i_k}\times\,\_\,$. Hence $t\circ d$ will be the sum of some $t(\alpha_{i_k}\times\,\_\,)$. Since $t$ is an arrow of $C$-modules, we have $t(\alpha_{i_k}\times\,\_\,)=\alpha_{i_k}\times t(\,\_\,)$. But, by the previous proposition and Lemma \ref{generado}, this equal to 0. Hence $t\circ d=0$ and, since this holds for any $x$, we have $\Phi=0$. This means that, for $1\leq p\leq n$, we have $\mathcal{H}H^p(A_n,\, A_n)=\mathcal{H}_{C}(K_p^n(C), A_n)$. \par
Now,  $K_p^n(C)=\mathop{\bigoplus}_{\substack{S\subseteq T\\ |S|=p}}\limits C$, with $T=\{1,\ldots ,n\}$, so, by Section 3 of \cite{primero}, we have 
\begin{displaymath}
\mathcal{H}_{C}\Bigg(\bigoplus_{\substack{S\subseteq T\\ |S|=p}} C,\, A_n\Bigg)\cong \bigoplus_{\substack{S\subseteq T\\ |S|=p}}\mathcal{H}_C(C,\, A_n)\cong\bigoplus_{\substack{S\subseteq T\\ |S|=p}} A_n=K_p^n(A_n). 
\end{displaymath}
Finally, by Section 6 of \cite{primero}, we know that $\mathcal{H}H^0(A_n,\, A_n)=A_n=K_{0}^n(A_n)$.
\end{proof}

In the next proposition we make use of some results concerning the tensor product over a monoid in $\mathcal{F}$,  which can be found in the Appendix.

\begin{prop}\label{hilbert} (Relative Hilbert syzygy Theorem) Let $A$ be a  tensor idempotent commutative monoid in $\CF$ and $t_1,\ldots, t_n$ be variables. Then, any $A[t_1,\ldots, t_n]$-module $M$ admits a finite $\CF$-split resolution
$$0\to L_n\to L_{n-1}\to\ldots\to L_1\to M\to 0$$
by $\CF$-projective $A[t_1,\ldots, t_n]$-modules.
\end{prop}
\begin{proof} We continue to denote $A_{2n}$ by $C$ and apply  the functor ${-}\otimes_{A_n}M$ to the resolution $K_*:=K_C(\alpha)_\mu$ of $A_n$ of Theorem~\ref{koszul An}.  Since $A_n\otimes_{A_n} M\cong M$ as $A_n$-modules, we get a complex
\[0\to K_n^n(C)\otimes_{A_n}M\to\ldots \to K_1^n(C)\otimes_{A_n}M\to C\otimes_{A_n}M\to M\to 0\tag{$K_*\otimes_{A_n}M$}\]
of $A_n$-modules.\par
 Consider the restriction from $(A_n,A_n)$-bimodules to right $A_n$-modules. We continue to denote this functor by $\mathcal{R}^C_{A_n}$, but taking into account that the $A_n$-modules are viewed as right modules.  By the proof of the previous theorem, we know that $\mathcal{R}^C_{A_n}(K_*)$ is a split exact complex.  Hence it is contractible and, since $\,\_\,\otimes_{A_n}M$ is an additive functor, the complex $(\Rres{C}{A_n}(K_*))\otimes_{A_n}M$ is contractible too. Now, by Lemma \ref{restricciones} of the Appendix, this complex is isomorphic to $\Rres{A_n}{I}(K_*\otimes_{A_n}M)=\mathcal{R}_{A_n}(K_*\otimes_{A_n}M)$. Hence, $\mathcal{R}_{A_n}(K_*\otimes_{A_n}M)$ is a split exact complex in $\mathcal{F}$. Finally, since  $\mathcal{R}_{A_n}$ is faithful and exact, it follows that the complex $(K_*\otimes_{A_n}M)$ is exact in $A_n\hbox{-Mod}$. Hence it is an $\CF$-split resolution of $M$.\medskip\par

Now, each $K_i^n(C)$, for $i>0$, is a direct sum of copies of the $(C,C)$-bimodule $C$. Also, $C\cong A_n\otimes A_n=A_n\otimes_{I}A_n$ (we notice in the Appendix that the tensor product over $I$ is the usual tensor product of $\mathcal{F}$), so
\[C\otimes_{A_n}M\cong (A_n\otimes_{I}A_n)\otimes_{A_n}M\cong  A_n\otimes_{I}(A_n\otimes_{A_n}M)\cong A_n\otimes_I M.\]
This shows in particular that $C\otimes_{A_n}M$ is $\CF$-projective, by Lemma~4.3 of~\cite{resmac}, since $A_n\otimes_IM$ is in the image of the left adjoint to the restriction functor from $A_n$-modules to $\CF$. Moreover,  $K_i^n(C)\otimes_{A_n}M\cong K_i^n(A_n)\otimes_I M$, so all the terms $K_i^n(C)\otimes_{A_n}M$, for $i>0$, are also $\CF$-projective. Hence, the complex $K_*\otimes_{A_n}M$, isomorphic to
\[0\to K_n^n(A_n)\otimes_{I}M\to\ldots \to K_1^n(A_n)\otimes_{I}M\to A_n\otimes_{I}M\to M\to 0,\]
is an $\CF$-split resolution of $M$ by $\CF$-projective $A_n$-modules. This completes the proof.
\end{proof}

\section{Appendix}

Let $A$ be a monoid in $\mathcal{F}$, $M$ a right A-module with action $\sigma_r:M\otimes A\rightarrow M$ and $N$ a left $A$-module with action $\sigma_l:A\otimes N\rightarrow N$. The \textit{tensor product over $A$ of $M$ and $N$}, denoted by $M\otimes_AN$,  is the coequalizer of the arrows given by the actions (see Section VII of \cite{maclane}),
\[
\xymatrix{
M\otimes A\otimes N\ar@<0.5ex>[rr]^-{\sigma_r\otimes N}\ar@<-0.5ex>[rr]_-{M\otimes \sigma_l}&& M\otimes N\ar[r] & M\otimes_AN,
}
\]
which exists since $\mathcal{F}$ is bicomplete. This construction defines an additive functor
\begin{displaymath}
\textrm{Mod-}A\, \times A\textrm{-Mod}\rightarrow \mathcal{F},\quad (M,\, N)\mapsto M\otimes_A N.
\end{displaymath}

Notice that, for $F$ and $T$ objects in $\mathcal{F}$, the tensor product over $I$, $F\otimes_{I}T$, is the usual tensor product of $\mathcal{F}$. 

\begin{rem}
Let $A$, $C$ and $D$ be monoids in $\mathcal{F}$, let $M$ be a $(C,\, A)$-bimodule and $N$ be an $(A,\, D)$-bimodule. The following facts are easily deduced from the definition.
\begin{enumerate}
\item If $Q$ is a $D$-module, then 
\begin{displaymath}
(M\otimes_AN)\otimes_D Q\cong M\otimes_A(N\otimes_DQ)
\end{displaymath}
and the isomorphism is natural in the three variables.

\item  $M\otimes_AN$ is a $(C,\, D)$-bimodule.

\item We have $M\otimes_AA\cong M$ and $A\otimes_AN\cong N$ in $\mathcal{F}$ and these are  natural isomorphisms of $C$-modules and right $D$-modules, respectively.
\end{enumerate}
\end{rem}

\begin{lema}
\label{restricciones}
Let $A$, $D$ and $E$ be monoids in $\mathcal{F}$, $M$ be a $(D,\, A)$-bimodule and $N$ be an $(A,\, E)$-bimodule. Let $\mathcal{R}^{D,\, A}_A$ denote the restriction functor $\mathcal{R}^{D\otimes A^{op}}_{I\otimes A^{op}}$ from $(D,A)$-bimodules to right $A$-modules and $\Rres{D,\,E}{E}$ be defined in a similar  way. Then, there is an isomorphism
$$\Rres{D,\,E}{E}(M\otimes_A N)\cong (\Rres{D,\,A}{A}M)\otimes_AN$$
of right $E$-modules, which is functorial in $M$ and $N$. 
\end{lema}
\begin{proof}
Indeed, by construction, we have an exact sequence
$$M\otimes A\otimes N\stackrel{\delta}{\to}M\otimes N\to M\otimes_AN\to 0$$
of $(D,E)$-bimodules, where $\delta=\sigma_r\otimes N-M\otimes\sigma_l$. Applying the functor $\Rres{D,\,E}{E}$ to this sequence gives the exact sequence
$$\Rres{D,\,E}{E}(M\otimes A\otimes N)\to\Rres{D,\,E}{E}(M\otimes N)\to \Rres{D,\,E}{E}(M\otimes_AN)\to 0.$$
Moreover, there are canonical isomorphisms of right $E$-modules
\begin{align*}
\Rres{D,\,E}{E}(M\otimes A\otimes N)&\cong (\Rres{D,\,A}{A}M)\otimes A\otimes N,\\
\Rres{D,\,E}{E}(M\otimes N)&\cong (\Rres{D,\,A}{A}M)\otimes N,
\end{align*}
so we get the exact sequence
$$ (\Rres{D,\,A}{A}M)\otimes A\otimes N\to(\Rres{D,\,A}{A}M)\otimes N\to \Rres{D,\,E}{E}(M\otimes_AN)\to 0.$$
Thus $\Rres{D,\,E}{E}(M\otimes_A N)\cong (\Rres{D,\,A}{A}M)\otimes_AN$, as claimed.
\end{proof}

%\bibliographystyle{abbrv}
%\bibliography{kosz}

\begin{thebibliography}{1}

\bibitem{resmac}
S.~Bouc.
\newblock R\'{e}solutions de foncteurs de {M}ackey.
\newblock In {\em Group representations: cohomology, group actions and topology
  ({S}eattle, {WA}, 1996)}, volume~63 of {\em Proc. Sympos. Pure Math.}, pages
  31--83. Amer. Math. Soc., Providence, RI, 1998.

\bibitem{biset}
S.~Bouc.
\newblock {\em Biset functors for finite groups}.
\newblock Springer, Berlin, 2010.

\bibitem{corr2}
S.~Bouc and J.~Th\'{e}venaz.
\newblock Correspondence functors and finiteness conditions.
\newblock {\em J. Algebra}, 495:150--198, 2018.

\bibitem{corfun-tensor}
S.~Bouc and J.~Th\'evenaz.
\newblock Tensor product of correspondence functors.
\newblock {\em J. Algebra}, 558:146--175, 2020.

\bibitem{maclane}
S.~Mac~Lane.
\newblock {\em Categories for the working mathematician}.
\newblock Springer, Berlin, 1971.

\bibitem{matsu}
H.~Matsumura.
\newblock {\em Commutative ring theory}, volume~8 of {\em Cambridge Studies in
  Advanced Mathematics}.
\newblock Cambridge University Press, Cambridge, 1986.
\newblock Translated from the Japanese by M. Reid.

\bibitem{lachica}
N.~Romero.
\newblock Simple modules over {G}reen biset functors.
\newblock {\em Journal of Algebra}, 367:203--221, 2012.

\bibitem{primero}
N.~Romero.
\newblock Hochschild cohomology for functors on linear symmetric monoidal
  categories.
\newblock {\em Ann. K-Theory}, 9(3):475--497, 2024.

\bibitem{weibel}
C.~A. Weibel.
\newblock {\em An introduction to homological algebra}, volume~38 of {\em
  Cambridge studies in advanced mathematics}.
\newblock Cambridge University Press, 1994.

\end{thebibliography}

\centerline{\rule{5ex}{.1ex}}
\begin{flushleft}
Serge Bouc, CNRS-LAMFA, Universit\'e de Picardie, 33 rue St Leu, 80039, Amiens, France.\\
{\tt serge.bouc@u-picardie.fr}\vspace{1ex}\\
Nadia Romero, DEMAT, UGTO, Jalisco s/n, Mineral de Valenciana, 36240, Guanajuato, Gto., Mexico.\\
{\tt nadia.romero@ugto.mx}
\end{flushleft}

\end{document}